\documentclass[12pt, a4paper]{amsart}
\usepackage{amssymb}
\usepackage{tikz} 
\usepackage{mathabx}
\usepackage{graphicx}
\usepackage{mathrsfs}
\usepackage{hyperref}
\usepackage{import}
\usepackage{tcolorbox}

\theoremstyle{plain}

\newtheorem{theorem}{Theorem}[section]
\newtheorem{lemma}[theorem]{Lemma}
\newtheorem{proposition}[theorem]{Proposition}
\newtheorem{corollary}[theorem]{Corollary}

\theoremstyle{definition}
\newtheorem{define}{Definition}[section]

\theoremstyle{remark}
\newtheorem{remark}[theorem]{Remark}

\def\dist{\text{dist}}

\newtheorem*{acknowledgement}{Acknowledgement}

\begin{document}

\date\today

\title[A note on exhaustion of hyperbolic complex manifolds]{A note on exhaustion of hyperbolic complex manifolds}
\author{Ninh Van Thu\textit{$^{1,2}$} and Trinh Huy Vu\textit{$^{1}$}}

\address{Ninh Van Thu}
\address{\textit{$^{1}$}~Department of Mathematics, Vietnam National University, Hanoi, 334 Nguyen Trai, Thanh Xuan, Hanoi, Vietnam}
\address{\textit{$^{2}$}~Thang Long Institute of Mathematics and Applied Sciences,
	Nghiem Xuan Yem, Hoang Mai, HaNoi, Vietnam}
\email{thunv@vnu.edu.vn}
\address{Trinh Huy Vu}
\address{\textit{$^{1}$}~Department of Mathematics, Vietnam
National University at Hanoi, 334 Nguyen Trai str., Hanoi, Vietnam}
\email{trinhhuyvu1508@gmail.com}


\begin{abstract}
The purpose of this article is to investigate a hyperbolic complex manifold $M$ exhausted by a pseudoconvex domain $\Omega$ in $\mathbb C^n$ via an exhausting sequence $\{f_j\colon \Omega\to M\}$ such that $f_j^{-1}(a)$ converges to a boundary point $\xi_0 \in \partial \Omega$ for some point $a\in M$.
\end{abstract}

\maketitle

\section{introduction}

Let $M$ and $\Omega$ be two complex manifolds. One says that \emph{$\Omega$ can exhaust $M$} or \emph{$M$ can be exhausted by $\Omega$} if for any compact subset $K$ of $M$ there is a holomorphic embedding $f_K \colon \Omega \to M$ such that $f_K(\Omega)\supset K$. In particular, one says that \emph{$M$ is a monotone union of $\Omega$} via a sequence of  holomorphic embeddings $f_j\colon \Omega\to M$ if $f_j(\Omega)\subset f_{j+1}(\Omega)$ for all $j$ and $M=\bigcup_{j=1}^\infty f_j(\Omega)$ (see \cite{FS77, Fr83}). 

In \cite[Theorem $1$]{Fr86}, there exists a bounded domain $D$ in $\mathbb C^n$ such that $D$ can exhaust any domain in $\mathbb C^n$. In addition, the unit ball $\mathbb B^n$ in $\mathbb C^n$ can exhaust many complex manifods, which are not biholomorphically equivalent to each other (see \cite{For04, FS77}). However, if $M$ in addition is hyperbolic then $M$ must be biholomorphically equivalent to $\mathbb B^n$ (cf. \cite{FS77}). Furthermore, any $n$-dimensional hyperbolic complex manifold, exhausted by a homogeneous bounded domain $D$ in $\mathbb C^n$, is biholomorphically equivalent to $D$. As a consequence, although the polydisc $\mathbb U^n$ and the unit ball $\mathbb B^n$ are both homogeneous and there is a domain $U$ in $\mathbb B^n$ that contains almost all of $\mathbb B^n$, i.e., $\mathbb B^n\setminus U$ has measure zero (cf. \cite[Theorem $1$]{FS77}) and is biholomorphically equivalent to $\mathbb U^n$, but $\mathbb U^n$ cannot exhaust the unit ball $\mathbb B^n$ since it is well-known that $\mathbb U^n$ is not biholomorphically equivalent to $\mathbb B^n$. 

Let $M$ be a hyperbolic complex manifold exhausted by a bounded domain $\Omega\subset \mathbb C^n$ via an exhausting sequence $\{f_j\colon \Omega\to M\}$. Let us fix a point $a\in M$. Then, thanks to the boundedness of $\Omega$, without loss of generality we may assume that $f_j^{-1}(a)\to p\in \overline{\Omega}$ as $j\to \infty$.  If $p\in \Omega$, then one always has $M$ is biholomorphically equivalent to $\Omega$ (cf. Lemma \ref{orbitinside} in Section \ref{S2}). 

The purpose of this paper is to investigate such a complex manifold $M$ with $p\in \partial \Omega$. More precisely, our first main result is the following theorem.
\begin{theorem}\label{togetmodel} Let $M$ be an $(n+1)$-dimensional hyperbolic complex  manifold  and let $\Omega$ be a pseudoconvex domain in $\mathbb{C}^{n+1}$ with $C^\infty$-smooth boundary. Suppose that $M$ can be exhausted by $\Omega$ via an exhausting sequence $\{f_j: \Omega \to M\}$. If there exists a point $a \in M$ such that the sequence $f_j^{-1}(a)$ converges $\Lambda$-nontangentially to a $h$-extendible boundary point $\xi_0 \in \partial \Omega$ (see Definition \ref{def-order} in Section \ref{S2} for definitions of the $\Lambda$-nontangentially convergence and of the $h$-extendibility), then $M$ is biholomorphically equivalent to the associated model $M_P$ for $\Omega$ at $\xi_0$.
\end{theorem}

When $\xi_0$ is a strongly pseudoconvex boundary point, we do not need the condition that the sequence $f_j^{-1}(a)$ converges $\Lambda$-nontangentially to $\xi_0$ as $j\to \infty$. Moreover, in this circumstance, the model $M_P$ is in fact biholomorphically equivalent to $M_{|z|^2}$, which is biholomorphically equivalent to the unit ball $\mathbb B^n$. More precisely, our second main result is the following theorem.
\begin{theorem}\label{togetmodelstronglypsc} Let $M$ be an $ (n+1) $-dimensional hyperbolic complex manifold and let $\Omega$ be a pseudoconvex domain in $\mathbb{C}^{n+1}$. Suppose that $\partial\Omega$ is $\mathcal{C}^2$-smooth boundary near a strongly pseudoconvex boundary point $\xi_0 \in \partial \Omega$. Suppose also that $M$ can be exhausted by $\Omega$ via an exhausting sequence $\{f_j: \Omega \to M\}$. If there exists a point $a \in M$ such that the sequence $\eta_j := f_j^{-1}(a)$ converges to $\xi_0$, then $M$ is biholomorphically equivalent to the unit ball $\mathbb{B}^{n+1}$.
\end{theorem}

Notice that Theorem \ref{togetmodelstronglypsc} is a local version of \cite[Theorem $1.1$]{DZ19} and \cite[Theorem I]{Fr83} (see Corollary \ref{str-psc-ex} in Section \ref{S3}). We note that their proofs are based on the boundary estimate of the Fridman invariant and of the squeezing function for strongly pseudoconvex domains. However, in order to prove Theorem \ref{togetmodel} and Theorem \ref{togetmodelstronglypsc}, we shall use the  scaling technique, achieved recently in \cite{Ber06, DN09, NN19}.

By applying Theorem \ref{togetmodelstronglypsc} and  Lemma \ref{orbitinside}, we also prove that if a hyperbolic complex manifold $M$ exhausted by a general ellipsoid $D_P$ (see Section \ref{S4} for the definition of $D_P$), then $M$ is either biholomorphically equivalent to $D_P$ or  the unit ball $\mathbb B^n$ (cf. Proposition \ref{generalellipsoid} in Section \ref{S4}).  In particular, when $D_P$ is an ellipsoid $E_m\; (m\in \mathbb Z_{\geq 1})$, given by
$$
E_m=\left\{(z,w)\in \mathbb C^2 \colon |w|^2+|z|^{2m}<1\right\}, 
$$
 in fact Proposition \ref{generalellipsoid} is a generalization of \cite[Theorem $1$]{Liu18}.  

The organization of this paper is as follows: In Section~\ref{S2} we provide some results concerning the normality of  a sequence of biholomorphisms and the $h$-extendibility. In Section \ref{S3}, we give our proofs of Theorem \ref{togetmodel} and Theorem \ref{togetmodelstronglypsc}. Finally, the proof of Proposition  \ref{generalellipsoid} will be introduced in Section \ref{S4}.

\section{The normality and the $h$-extendibility}\label{S2}

\subsection{The normality of a sequence of biholomorphisms}

First of all, we recall the following definition (see \cite{GK} or \cite{DN09}). 
\begin{define} Let $\{\Omega_i\}_{i=1}^\infty$ be a sequence of open sets in a complex manifold $M$ and $\Omega_0 $ be an open set of $M$. The sequence $\{\Omega_i\}_{i=1}^\infty$ is said to converge to $\Omega_0 $ (written $\lim\Omega_i=\Omega_0$) if and only if 
	\begin{enumerate}
		\item[(i)] For any compact set $K\subset \Omega_0,$ there is an $i_0=i_0(K)$ such that $i\geq i_0$ implies that $K\subset \Omega_i$; and 
		\item[(ii)] If $K$ is a compact set which is contained in $\Omega_i$ for all sufficiently large $i,$ then  $K\subset \Omega_0$.
	\end{enumerate}  
\end{define}

Next, we recall the following proposition, which is a generalization of the theorem of H. Cartan (see \cite{DN09, GK, TM}).
\begin{proposition} \label{T:7}  Let $\{A_i\}_{i=1}^\infty$ and $\{\Omega_i\}_{i=1}^\infty$ be sequences of domains in a complex manifold $M$ with $\lim A_i=A_0$ and $\lim \Omega_i=\Omega_0$ for some (uniquely determined) domains $A_0$, $\Omega_0$ in $M$. Suppose that $\{f_i: A_i \to \Omega_i\} $ is a sequence of biholomorphic maps. Suppose also that the sequence $\{f_i: A_i\to M \}$ converges uniformly on compact subsets of $ A_0$ to a holomorphic map $F:A_0\to M $ and the sequence $\{g_i:=f^{-1}_i: \Omega_i\to M \}$ converges uniformly on compact subsets of $\Omega_0$ to a holomorphic map $G:\Omega_0\to M $.  Then either of the following assertions holds. 
	\begin{enumerate}
		\item[(i)] The sequence $\{f_i\}$ is compactly divergent, i.e., for each compact set $K\subset A_0$ and each compact set $L\subset \Omega_0$, there exists an integer $i_0$ such that $f_i(K)\cap L=\emptyset$ for $i\geq i_0$; or
		\item[(ii)] There exists a subsequence $\{f_{i_j}\}\subset \{f_i\}$  such that the sequence $\{f_{i_j}\}$ converges uniformly on compact subsets of $A_0$ to a biholomorphic map $F: A_0 \to \Omega_0$.
	\end{enumerate}
\end{proposition}

\begin{remark} \label{r1}  By \cite[Proposition $2.1$]{Ber94} or \cite[Proposition $2.2$]{DN09} and by the hypotheses of Theorem \ref{togetmodel} and Theorem \ref{togetmodelstronglypsc}, it follows that for each compact subset $K\Subset M$ and each neighborhood $U$ of $\xi_0$ in $\mathbb C^{n+1}$, there exists an integer $j_0=j_0(K)$ such that $K\subset f_j(\Omega\cap U)$ for all $j\geq j_0$. Consequently, the sequence of domains $\{f_j(\Omega\cap U)\}$ converges to $M$. 
\end{remark}

We will finish this subsection by recalling the following lemma (cf. \cite[Lemma $1.1$]{Fr83}).
\begin{lemma}[see \cite{Fr83}]\label{orbitinside}
Let $M$ be a hyperbolic manifold of complex dimension $n$. Assume that $M$ can be exhausted by $\Omega$ via an exhausting sequence $\{f_j: \Omega \to M\}$, where $\Omega$ is a bounded domain in $\mathbb{C}^n$. Suppose that there is an interior point $a \in M$ such that $f_j^{-1} (a) \to p \in \Omega$. Then, $M$ is biholomorphically equivalent to $\Omega$.
\end{lemma}

\subsection{The $h$-extendibility }
In this subsection, we recall some definitions and notations given in \cite{Cat84, Yu95}.

Let $\Omega$ be a smooth pseudoconvex domain in $\mathbb C^{n+1}$ and $p\in \partial\Omega$. Let $\rho$ be a local defining function for $\Omega$ near $p$. Suppose that the multitype $\mathcal{M}(p)=(1,m_1,\ldots,m_n)$ is finite. (See \cite{Cat84} for the notion of multitype.)  Let us denote by $\Lambda=\left(1/m_1,\ldots,1/m_n\right)$.  Then, there are distinguished coordinates $(z,w)=(z_1,\ldots,z_n,w)$ such that $p=0$ and $\rho(z,w)$ can be expanded near $0$ as follows:
$$
\rho(z,w)=\mathrm{Re}(w)+P(z)+R(z,w),
$$ 
 where $P$ is a $\Lambda$-homogeneous plurisubharmonic polynomial that contains no pluriharmonic terms, $R$ is smooth and satisfies 
 $$
 |R(z,w)|\leq C \left( |w|+ \sum_{j=1}^n |z_j|^{m_j} \right)^\gamma,
 $$ 
 for some constant $\gamma>1$ and $C>0$. Here and in what follows, a polynomial $P$ is called $\Lambda$-homogeneous if 
 $$
 P(t^{1/m_1}z_1,t^{1/m_2}z_2, \ldots,t^{1/m_n}z_n)=P(z),\; \forall t>0, \forall z\in \mathbb C^n. 
 $$
 
\begin{define}[see \cite{NN19}]\label{def-order} The domain $M_P=\{(z,w)\in \mathbb C^n\times \mathbb C\colon \mathrm{Re}(w)+P(z)<0\}$ is called an \emph{associated model} of $\Omega$ at $p$.  A boundary point $p\in \partial \Omega$ is called \emph{$h$-extendible} if its associated model $M_P$ is \emph{$h$-extendible}, i.e., $M_P$ is of finite type (see \cite[Corollary $2.3$]{Yu94}). In this circumstance, we say that a sequence $\{\eta_j=(\alpha_j,\beta_j)\}\subset  \Omega$ \emph{converges $\Lambda$-nontangentially to $p$} if $|\mathrm{Im}(\beta_j)|\lesssim |\mathrm{dist}(\eta_j,\partial \Omega)|$ and $ \sigma(\alpha_j) \lesssim |\mathrm{dist}(\eta_j,\partial \Omega)|$, where 
$$
\sigma(z)=\sum_{k=1}^n |z_k|^{m_k}.
$$ 
\end{define}
Throughout this paper, we use $\lesssim$ and $\gtrsim$ to denote inequalities up to a positive multiplicative constant. Moreover, we use $\approx $ for the combination of $\lesssim$ and $\gtrsim$. In addition, $\mathrm{dist}(z,\partial\Omega)$ denotes the Euclidean distance from $z$ to $\partial\Omega$. Furthermore, for $\mu>0$ we denote by $\mathcal{O}(\mu,\Lambda)$ the set of  all smooth functions $f$ defined near the origin of $\mathbb C^n$ such that
$$
D^\alpha \overline{D}^\beta f(0)=0~\text{whenever}~ \sum_{j=1}^n (\alpha_j+\beta_j)\dfrac{1}{m_j} \leq \mu.
$$
If $n=1$ and $\Lambda = (1)$ then we use $\mathcal{O}(\mu)$ to denote the functions vanishing to order at least $\mu$ at the origin (cf. \cite{Cat84, Yu95}).

\section{Proofs of Theorem \ref{togetmodel} and Theorem \ref{togetmodelstronglypsc}}\label{S3}

This section is devoted to our proofs of Theorem \ref{togetmodel} and Theorem \ref{togetmodelstronglypsc}. First of all, let us recall the definition of the Kobayashi infinitesimal pseudometric and the Kobayashi pseudodistance as follows:
\begin{define} Let $M$ be a complex manifold. The Kobayashi infinitesimal pseudometric $F_M \colon M\times T^{1,0}M\to \mathbb R$ is defined by
	$$
	F_M(p,X)=\inf \left\{c>0\;|\; \exists \; f\colon \Delta \to M \;\text{holomorphic with}\; f(0)=p, f'(0)=X/c \right\},
	$$
for any $p\in M$ and $X\in T^{1,0}M$, where $\Delta $ is the unit open disk of $\mathbb C$. Moreover, the Kobayashi pseudodistance $d_M^K\colon M\times M \to \mathbb R$ is defined by
	$$
	d_M^K(p,q)=\inf_\gamma\int_0^1 F_M(\gamma(t),\gamma'(t)) dt,
	$$
for any $p,q\in M$ where the infimum is taken over all differentiable curves $\gamma:[0,1] \to M$ joining $p$ and $q$. A complex manifold $M$ is called hyperbolic if $d_M^K(p,q)$ is actually a distance, i.e., $d_M^K(p,q)>0$ whenever $p\ne q$. 
\end{define}

Next, we need the following lemma, whose proof will be given in Appendix for the convenience of the reader, and the following proposition.
\begin{lemma}\label{conti-kob} Assume that $\{D_j\}$ is a sequence of domains in $\mathbb C^{n+1}$ converging to a model $M_P$ of finite type. Then, we have
	$$
	\lim_{j\to \infty} F_{D_j}(z,X)=F_{M_P}(z,X),~\forall (z,X)\in M_P\times \mathbb C^{n+1}.
	$$
	Moreover, the convergence takes place uniformly over compact subsets of $M_P\times \mathbb C^{n+1}$.
\end{lemma}
 \begin{proposition}[see \cite{NN19}]\label{pro-scaling} Assume that $\{D_j\}$ is a sequence of domains in $\mathbb C^{n+1}$ converging to a model $M_P$ of finite type. Assume also that $\omega$ is a domain in $\mathbb C^k$ and $\sigma_j: \omega \to D_j$ is a sequence of holomorphic mappings such that $\{\sigma_j(a)\}\Subset M_P$ for some $a\in \omega$. Then $\{\sigma_j\}$ contains a subsequence that converges locally uniformly to a holomorphic map $\sigma: \omega \to M_P$. 
\end{proposition} 

Now we are ready to prove Theorem \ref{togetmodel} and Theorem \ref{togetmodelstronglypsc}.
\begin{proof}[Proof of Theorem \ref{togetmodel}]
Let $\rho$ be a local defining function for $\Omega$ near $\xi_0$ and the multitype $\mathcal{M}(\xi_0)=(1,m_1,\ldots,m_n)$ is finite. In what follows, denote by $\Lambda=(1/m_1,\ldots,1/m_n)$. Since $\xi_0$ is a $h$-extendible point, there exist local holomorphic coordinates $(z,w)$ in which $\xi_0=0$ and $\Omega$ can be described in a neighborhood $U_0$ of $0$ as follows: 
 $$
 \Omega\cap U_0=\left\{\rho(z,w)=\mathrm{Re}(w)+ P(z) +R_1(z) + R_2(\mathrm{Im} w)+(\mathrm{Im} w) R(z)<0\right\},
 $$ 
where $P$ is a $\Lambda$-homogeneous plurisubharmonic real-valued polynomial containing no pluriharmonic terms, $R_1\in \mathcal{O}(1, \Lambda),R\in \mathcal{O}(1/2, \Lambda) $, and $R_2\in \mathcal{O}(2)$. (See the proof of Theorem $1.1$ in \cite{NN19} or  the proof of Lemma $4.11$ in \cite{Yu95}.)
 
By assumption, there exists a point $a\in M$ such that the sequence $\eta_j:=f^{-1}_j(a)$ converges $\Lambda$-nontangentially to $\xi_0$. Without loss of generality, we may assume that the sequence $\{\eta_j\}\subset \Omega\cap U_0$ and we write $\eta_j=(\alpha_j,\beta_j)=(\alpha_{j1},\ldots,\alpha_{jn},\beta_j)$ for all $j$. Then, the sequence $\{\eta_j:=f^{-1}(a)\}$ has the following properties: 
\begin{itemize}
\item[(a)] $|\mathrm{Im}(\beta_j)|\lesssim |\mathrm{dist}(\eta_j,\partial \Omega)|$;
\item[(b)] $|\alpha_{jk}|^{m_k}\lesssim |\mathrm{dist}(\eta_j,\partial \Omega)|$ for $1\leq k\leq n$.
\end{itemize}

For the sequence $\{\eta_j=(\alpha_j,\beta_j)\}$, we associate with a sequence of points $\eta_j'=(\alpha_{j1}, \ldots, \alpha_{jn},\beta_j +\epsilon_j)$, where $\epsilon_j>0$, such that $\eta_j'$ is in the hypersurface $\{\rho=0\}$ for all $j$. We note that $\epsilon_j\approx \mathrm{dist}(\eta_j,\partial \Omega)$. Now let us consider the sequences of dilations $\Delta^{\epsilon_j}$ and translations $L_{\eta_j'}$, defined respectively by
$$
\Delta^{\epsilon_j}(z_1,\ldots,z_n,w)=\left(\frac{z_1}{\epsilon_j^{1/m_1}},\ldots,\frac{z_n}{\epsilon_j^{1/m_n}},\frac{w}{\epsilon_j}\right)
$$
and 
$$
L_{\eta_j'}(z,w)=(z,w)-\eta'_j=(z-\alpha'_j,w-\beta'_j).
$$
Under the change of variables $(\tilde z,\tilde w):=\Delta^{\epsilon_j}\circ L_{\eta'_j}(z,w)$, i.e.,
\[
\begin{cases}
w-\beta'_j= \epsilon_j\tilde{w}\\
z_k-\alpha'_{j k}=\epsilon_j^{1/m_k}\tilde{z}_k,\, k=1,\ldots,n,
\end{cases}
\]
one can see that $\Delta^{\epsilon_j}\circ L_{\eta_j'}(\alpha_j,\beta_j)=(0,\cdots,0,-1)$ for all $j$. Moreover, as in \cite{NN19}, after taking a subsequence if necessary, we may assume that the sequence of domains $\Omega_j:=\Delta^{\epsilon_j}\circ L_{\eta_j'}(\Omega\cap U_0) $ converges to the following model
$$
M_{P,\alpha}:=\left \{(\tilde z,\tilde w)\in \mathbb C^n\times\mathbb C\colon \mathrm{Re}(\tilde w)+P(\tilde z+\alpha)-P(\alpha)<0\right\},
$$
which is obviously biholomorphically equivalent to the model $M_P$. Without loss of generality, in what follows we always assume that $\{\Omega_j\}$ converges to $M_P$. 

Now we first consider the sequence of biholomorphisms $F_j:= T_j\circ f_j^{-1}\colon M\supset f_j(\Omega\cap U_0)\to \Omega_j$, where $T_j:=\Delta^{\epsilon_j}\circ L_{\eta_j'}$. Since $F_j(a)=(0',-1)$ and notice that $f_j(\Omega\cap U_0)$ converges to $M$ as $j\to \infty$ (see Remark \ref{r1}), by Proposition \ref{pro-scaling}, without loss of generality, we may assume that the sequence $F_j$ converges uniformly on on every compact subset of $M$ to a holomorphic map $F$ from $M$ to $\mathbb C^{n+1}$. Note that $F(M)$ contains a neighborhood of $(0',-1)$ and $F(M)\subset \overline{M_P}$. 

 Since  $F_j$ is normal, by the Cauchy theorem it follows that $\{J(F_j)\}$ converges uniformly on every compact subsets of $M$ to $J(F)$, where $J(F)$ denotes the Jacobian determinant of $F$. However, by the Cartan theorem, $J(F_j)(z)$ is nowhere zero for any $j$ because $F_j$ is a biholomorphism. Then, the Hurwitz theorem implies that $J(F)$ is a zero function or nowhere zero. In the case that $JF\equiv 0$, $F$ is regular at no point of $M$. As $F(M)$ contains a neighborhood of $(0',-1)$, the Sard theorem shows that $F$ is regular outside a proper subvariety of $M$, which is a contradiction. This yields $JF$ is nowhere zero and  hence $F$ is regular everywhere on $M$. By \cite[Lemma 0]{FS77}, it follows that $F(M)$ is open and $F(M)\subset M_P$.
 
 Next, we shall prove that $F$ is one-to-one. Indeed, let $z_1, z_2\in M$ be arbitrary. Fix a compact subset $L \Subset M$ such that $z_1,z_2\in L$. Then, by Remark \ref{r1} there is a $j_0(L)>0$ such that $L\subset f_j(\Omega\cap U_0)$ and $F_j(L)\subset K\Subset M_P$ for all $j>j_0(L)$, where $K$ is a compact subset of $M_P$. By Lemma \ref{conti-kob} and the decreasing property of Kobayashi distance, one has
 \begin{align*}
d^K_M(z_1,z_2)&\leq d^K_{f_j(\Omega\cap U_0)}(z_1,z_2)=d^K_{\Omega_j}(F_j(z_1),F_j(z_2)))\leq C \cdot d^K_{M_P}(F_j(z_1),F_j(z_2))\\
&\leq C \left( d^K_{M_P}(F(z_1),F(z_2))+ d^K_{M_P}(F_j(z_1),F(z_1))+d^K_{M_P}(F_j(z_2),F(z_2))\right),
\end{align*}
where $C>1$ is a positive constant. Letting $j\to \infty$, we obtain 
$$
d^K_M(z_1,z_2)\leq  C \cdot d^K_{M_P}(F(z_1),F(z_2)). 
$$
Since $M$ is hyperbolic, it follows that if $F(z_1)=F(z_2)$, then $z_1=z_2$. Consequently, $F$ is one-to-one, as desired.

Finally, because of the biholomorphism from $M$  to $F(M)\subset M_P$ and the tautness of $M_P$ (cf. \cite{Yu95}),  it follows that the sequence $ F_j^{-1}=f_j\circ T_j^{-1} \colon T_j(\Omega \cap U_0)\to   f_j(\Omega \cap U_0) \subset M$ is also normal. Moreover, since $T_j\circ f_j^{-1}(a)=(0',-1)\in M_P$, it follows that the sequence $T_j\circ f_j^{-1}$ is not compactly divergent. Therefore, by Proposition \ref{T:7}, after taking some subsequence we may assume that $T_j\circ f_j^{-1}$ converges uniformly on every compact subset of $M$ to a biholomorphism from $M$ onto $M_P$. Hence, the proof is complete.
\end{proof}
\begin{remark}
If $M$ is a bounded domain in $\mathbb C^{n+1}$, the normality of the sequence $F_j^{-1}$ can be shown by using the Montel theorem. Thus, the proof of Theorem \ref{togetmodel} simply follows from Proposition \ref{T:7}.  
\end{remark}

\begin{proof}[Proof of Theorem \ref{togetmodelstronglypsc}]
Let $\rho$ be a local defining function for $\Omega$ near $\xi_0$. We may assume that $\xi_0=0$.  After a linear change of coordinates, one can find local holomorphic coordinates $(\tilde z,\tilde w)=(\tilde z_1,\cdots, \tilde z_n,\tilde w)$, defined on a neighborhood $U_0$ of $\xi_0$, such that 
\begin{equation*}
\begin{split}
\rho(\tilde z,\tilde w)=\mathrm{Re}(\tilde w)+ \sum_{j=1}^{n}|\tilde z_j|^2+ O(|\tilde w| \|\tilde z\|+\|\tilde z\|^3)
\end{split}
\end{equation*}

By \cite[Proposition 3.1]{DN09} (or Subsection $3.1$ in \cite{Ber06} for the case $n=1$), for each point $\eta$ in a small neighborhood of the origin, there exists an automorphism $\Phi_\eta$ of $\mathbb C^n$ such that
\begin{equation*} 
\begin{split}
\rho(\Phi_{\eta}^{-1}(z,w))-\rho(\eta)=\mathrm{Re}(w)+  \sum_{j=1}^{n}|z_j|^2+ O(|w| \|z\|+\|z\|^3).
\end{split}
\end{equation*}
 Let us define an anisotropic dilation $\Delta^\epsilon$ by 
$$
\Delta^\epsilon (z_1,\cdots,z_n,w)= \left(\frac{z_1}{\sqrt{\epsilon}},\cdots,\frac{z_{n}}{\sqrt{\epsilon}},\frac{w}{\epsilon}\right).
$$
For each $\eta\in \partial \Omega$, if we set $\rho_\eta^\epsilon(z, w)=\epsilon^{-1}\rho\circ \Phi_\eta^{-1}\circ( \Delta^\epsilon)^{-1}(z,w)$, then 
\begin{equation*}
\rho_\eta^\epsilon(z, w)= \mathrm{Re}(w)+\sum_{j=1}^{n}|z_j|^2+O(\sqrt{\epsilon}).
\end{equation*}

By assumption, the sequence $\eta_j:=f^{-1}_j(a)$ converges to $\xi_0$. Then, we associate with a sequence of points  ${\eta}_j'=(\eta_{j1}, \cdots, \eta_{jn},\eta_{j(n+1)}+\epsilon_j)$, $ \epsilon_j>0$, such that ${\eta}_j'$ is in the hypersurface $\{\rho=0\}$. Then $ \Delta^{\epsilon_j}\circ \Phi_{{\eta'}_j}({\eta}_j)=(0,\cdots,0,-1)$ and one can see that $ \Delta^{\epsilon_j}\circ \Phi_{{\eta'}_j}(\{\rho=0\}) $ is defined by an equation of the form
\begin{equation*}
\begin{split}
\mathrm{Re}(w)+\sum_{j=1}^{n}|z_j|^2+O(\sqrt{\epsilon_j})=0.
\end{split}
\end{equation*}
Therefore, it follows that, after taking a subsequence if necessary, $\Omega_j:=\Delta^{\epsilon_j}\circ \Phi_{{\eta'}_p}(U_0^-)$ converges to the following domain
\begin{equation}\label{Eq29} 
\mathcal{E}:=\{\hat\rho:= \mathrm{Re}(w)+\sum_{j=1}^{n}|z_j|^2<0\},
\end{equation}
which is biholomorphically equivalent to the unit ball $\mathbb B^{n+1}$. 

Now let us consider the sequence of biholomorphisms $F_j:= T_j\circ f_j^{-1} \colon M\supset f_j(\Omega \cap U_0)\to T_j(\Omega \cap U_0)$, where $T_j:= \Delta^{\epsilon_j}\circ \Phi_{{\eta'}_j}$. Since $F_j(a)=(0',-1)$, by \cite[Theorem 3.11]{DN09}, without loss of generality, we may assume that the sequence $F_j$ converges uniformly on every compact subset of $M$ to a holomorphic map $F$ from $M$ to $\mathbb C^{n+1}$. Note that $F(M)$ contains a neighborhood of $(0',-1)$ and $F(M)\subset \overline{M_P}$. Following the argument as in the proof of Theorem \ref{togetmodel}, we conclude that  $F$ is a biholomorphism from $M$ onto $\mathcal{E}$, and thus  $M$ is biholomorphically equivalent to $\mathbb B^{n+1}$, as desired.
\end{proof}

By Lemma \ref{orbitinside} and Theorem \ref{togetmodelstronglypsc}, we obtain the following corollary, proved by F. S. Deng and X. J. Zhang \cite[Theorem 2.4]{DZ19} and by B. L. Fridman \cite[Theorem I]{Fr83}.
\begin{corollary} \label{str-psc-ex} Let $D$ be a bounded strictly pseudoconvex domain in $\mathbb C^n$ with $\mathcal{C}^2$-smooth boundary. If a bounded domain $\Omega$ can be exhausted by $D$, then $\Omega$ is biholomorphically equivalent to $D$ or the unit ball $\mathbb B^n$.
\end{corollary}

\section{Exhausting a complex manifold by a general ellipsoid}\label{S4}

In this section, we are going to prove that if a complex manifold $M$ can be exhausted by a general ellipsoid $D_P$ (see the definition of $D_P$ below), then $M$ is biholomorphically equivalent to either $D_P$ or the unit ball $B^n$. 

First of all, let us fix $n$ positive integers $m_1,\ldots, m_{n-1}$ and denote by $\Lambda:=\left(\frac{1}{m_1}, \ldots, \frac{1}{m_{n-1}}\right)$. We assign weights $\frac{1}{m_1}, \ldots, \frac{1}{m_{n-1}}, 1$ to $z_1,\ldots,z_n$. For an $(n-1)$-tuple $K = (k_1,\ldots,k_{n-1}) \in \mathbb{Z}^{n-1}_{\geq 0}$, denote the weight of $K$ by $$wt(K) := \sum_{j=1}^{k-1} \dfrac{k_j}{m_j}.$$

Next, we consider the general ellipsoid $D_P$ in $\mathbb C^n\;(n\geq2)$, defined by
\begin{equation*}
\begin{split}
D_P &:=\{(z',z_n)\in \mathbb C^n\colon |z_n|^2+P(z')<1\},
\end{split}
\end{equation*}
where 
\begin{equation}\label{series expression of P on D_P}
P(z')=\sum_{wt(K)=wt(L)=1/2} a_{KL} {z'}^K  \bar{z'}^L,
\end{equation}
where $a_{KL}\in \mathbb C$ with $a_{KL}=\bar{a}_{LK}$, satisfying that $P(z')>0$ whenever $z' \in \mathbb{C}^{n-1} \setminus \{0'\}$. We would like to emphasize here that the polynomial $P$ given in (\ref{series expression of P on D_P}) is $\Lambda$-homogeneous and the assumption that $P(z')>0$ whenever $z'\ne 0$ ensures that $D_P$ is bounded in $\mathbb{C}^n$ (cf. \cite[Lemma 6]{NNTK19}). Moreover, since $P(z')>0$ for $z'\ne 0$ and by the $\Lambda$-homogeneity, there are two constants $c_1,c_2>0$ such that
$$
c_1 \sigma_\Lambda(z') \leq P(z')\leq c_2 \sigma_\Lambda(z'), \; \forall z'\in \mathbb C^{n-1},
$$
where $\sigma_\Lambda(z')=|z_1|^{m_1}+\cdots+|z_{n-1}|^{m_{n-1}}$. In addition, $D_P$ is called a WB-domain if it is strongly pseudoconvex at every boundary point outside the set $\{(0',e^{i\theta})\colon \theta\in \mathbb R\}$ (cf. \cite{AGK16}).

Now we prove the following proposition.

\begin{proposition}  \label{generalellipsoid} 	Let $M$ be a $n$-dimensional complex hyperbolic manifold. Suppose that $M$ can be exhausted by the general ellipsoid $D_P$ via an exhausting sequence $\{f_j: D_P \to M\}$. If $D_P$ is a $WB$-domain, then $M$ is biholomorphically equivalent to either $D_P$ or the unit ball $\mathbb{B}^n$.
\end{proposition}

\begin{remark} 
The possibility that $M$ is biholomorphic onto the unit ball $\mathbb B^n$ is not excluded because  $D_P$ can exhaust the unit ball $\mathbb B^n$ by \cite[Corollary $1.4$]{FM95}.
\end{remark}
\begin{proof}[Proof of Proposition \ref{generalellipsoid}]
Let $q$ be an arbitrary point in $M$. Then, thanks to the boundedness of $D_P$, after passing to a subsequence if necessary we may assume that the sequence $\{f_j^{-1}(q)\}_{j=1}^{\infty}$ converges to a point $p\in \overline{D_P}$ as $j \to \infty$. 

\smallskip
We now divide the argument into two cases as follows:

\smallskip
\noindent
{\bf Case 1.} $f_j^{-1}(q)\to p\in D_P$. Then, it follows from Lemma \ref{orbitinside} that $M$ is biholomorphically equivalent to $D_P$.

\medskip

\noindent
{\bf Case 2.}  $f_j^{-1}(q)\to p\in\partial D_P$. Let us write $f_j^{-1}(q)=(a_j', a_{jn})\in D_P$ and $p=(a',a_n)\in \partial D_P$. As in \cite{NNTK19}, for each $j\in \mathbb N^*$ we consider $\psi_j\in \mathrm{Aut}(D_P)$, defined by  
$$
\psi_j(z)=\left( \frac{(1-|a_{jn}|^2)^{1/m_1}}{(1-\bar{a}_{jn}z_n)^{2/m_1}} z_1,\ldots,  \frac{(1-|a_{jn}|^2)^{1/m_{n-1}}}{(1-\bar{a}_{jn}z_n)^{2/m_{n-1}}} z_{n-1}, \frac{z_n-a_{jn}}{1-\bar{a}_{jn} z_n}\right).
$$
Then $\psi_j\circ f_j(q)=(b_j,0)$, where 
$$
b_j=  \left( \frac{a_{j1}}{(1-|a_{jn}|^2)^{1/m_1}} ,\ldots, \frac{a_{j (n-1)}}{(1-|a_{jn}|^2)^{1/m_{n-1}}}\right),\; \forall  j\in \mathbb N^*. 
$$
Without loss of generality, one may assume that $b_j\to b\in \mathbb C^{n-1}$ as $j\to \infty$.
\medskip

Since $D_P$ is a $WB$-domain, two possibilities may occur:

\smallskip
\noindent
{\bf Subcase 1:} $p=(a',a_n)$ is a strongly pseudoconvex boundary point.  In this subcase, it follows directly from Theorem \ref{togetmodelstronglypsc} that $M$ is biholomorphically equivalent to $\mathbb B^{n}$.

\medskip

\noindent
{\bf Subcase 2:} $p=(0',e^{i\theta})$ is a weakly pseudoconvex boundary point. In this subcase, one must have $a_j'\to 0'$ and $a_{jn}\to e^{i\theta}$ as $j\to \infty$. Denote by $\rho(z):=|z_n|^2-1+P(z')$ a definition function for $D_P$. Then $\dist(a_j, \partial D_P)\approx -\rho(a_j)= 1-|a_{jn}|^2-P(a_j')$.
Suppose that $\{a_j\}$ converges $\Lambda$-nontangentially to $p$, i.e., $P(a_j')\approx \sigma_\Lambda(a_j')\lesssim \dist(a_j, \partial D_P)$, or equivalently $P(a_j')\leq C(1-|a_{jn}|^2-P(a_j')),\; \forall j\in \mathbb N^*$, for some $C>0$. This implies that 
$$
P(a_j')\leq \dfrac{C}{1+C}(1-|a_{jn}|^2),\; \forall j\in \mathbb N^*,
$$
and thus $P(b_j)=\dfrac{1}{1-|a_{jn}|^2}P(a_j')\leq \dfrac{C}{1+C}<1,\; \forall j\in \mathbb N^*$. This yields $ \psi_j\circ f_j^{-1}(q)=(b_j,0)\to (b,0)\in D_P$ as $j\to \infty$. So, again by Lemma \ref{orbitinside} one concludes that $M$ is biholomorphically equivalent to $D_P$.

Now let us consider the case that the sequence $\{a_j\}$ does not converge $\Lambda$-nontangentially to $p$, i.e., $P(a_j')\geq c_j  \dist(a_j, \partial D_P), \; \forall j\in \mathbb N^*$, where $0<c_j\to +\infty$. This implies that $P(a_j')\geq c_j'(1-|a_{jn}|^2-P(a_j')),\; \forall j\in \mathbb N^*$, for some $0<c_j'\to +\infty$, and hence
$$
P(a_j')\geq \dfrac{c_j'}{1+c_j'}(1-|a_{jn}|^2),\; \forall j\in \mathbb N^*.
$$
Thus, one obtains that $P(b_j)=\dfrac{1}{1-|a_{jn}|^2}P(a_j')\geq \dfrac{c_j'}{1+c_j'}$, which implies that $P(b)=1$. Consequently,  $\psi_j\circ f_j^{-1}(q)$ converges to the strongly pseudoconvex boundary point $p'=(b,0)$ of $\partial D_P$. Hence, as in Subcase $1$, it follows from Theorem \ref{togetmodelstronglypsc} that $M$ is biholomorphically equivalent to $\mathbb B^{n}$.

Therefore, altogether, the proof of Proposition \ref{generalellipsoid} finally follows.
\end{proof}

\section*{Appendix}

\begin{proof}[Proof of Lemma \ref{conti-kob}]
We shall follow the proof of \cite[Theorem $2.1$]{Yu95} with minor modifications. To do this, let us fix compact subsets $K\Subset M_P$ and $L\Subset \mathbb C^{n+1}$. Then it suffices to prove that $F_{D_j}(z,X)$ converges to $F_{M_P}(z,X)$ uniformly on $K\times L$. Indeed, suppose otherwise. Then, there exist $\epsilon_0>0$, a sequence of points $\{z_{j_\ell}\}\subset K$, and a sequence $X_{j_\ell}\subset L$ such that
$$
|F_{D_{j_\ell}}(z_{j_\ell},X_{j_\ell})-F_{M_P}(z_{j_\ell},X_{j_\ell})|>\epsilon_0,~\forall~\ell\geq 1.
$$

By the homogeneity of the Kobayashi metrics $F(z,X)$ in $X$, we may assume that $\|X_{j_\ell}\|=1$ for all $\ell\geq 1$. Moreover, passing to subsequences, we may also assume that $z_{j_\ell}\to z_0\in K$ and $X_{j_\ell}\to X_0\in L$ as $\ell \to \infty$. Since $M_P$ is taut (see \cite[Theorem $3.13$]{Yu95}), for each $(z,X)\in M_P\times \mathbb C^{n+1}$ with $X\ne 0$, there exists an analytic disc $\varphi\in \mathrm{Hol}(\Delta, M_P)$ such that $\varphi(0)=z$ and $\varphi'(0)=X/F_{M_P}(z,X)$. This implies that $F_{M_P}(z,X)$ is continuous on $M_P\times \mathbb C^{n+1}$. Hence, we obtain 
$$
F_{M_P}(z_{j_\ell},X_{j_\ell})\to F_{M_P}(z_0,X_0),
$$
and thus we have
\begin{align}\label{eq136-0}
|F_{D_{j_\ell}}(z_{j_\ell},X_{j_\ell})-F_{M_P}(z_0,X_0)|>\epsilon_0/2
\end{align}
for $\ell$ big enough.

By definition, for any $\delta\in (0,1)$ there exists a sequence of analytic discs $\varphi_{j_\ell}\in \mathrm{Hol}(\Delta, D_{j_\ell})$ such that $\varphi_{j_\ell}(0)=z_0,\varphi_{j_\ell}'(0)= \lambda_{j_\ell} X_{j_\ell}$, where $\lambda_{j_\ell}>0$, and 
\begin{align*}
F_{D_{j_\ell}}(z_{j_\ell},X_{j_\ell})\geq \frac{1}{\lambda_{j_\ell}}-\delta.
\end{align*}

It follows from Proposition \ref{pro-scaling} that every subsequence of the sequence $\{\varphi_{j_\ell}\}$ has a subsequence converging to some analytic disc $\psi\in \mathrm{Hol}(\Delta, M_P)$ such that $\psi(0)=z_0,\psi'(0)= \lambda X_0$, for some $\lambda>0$. Thus, one obtains that
$$
F_{M_P}(z_0,X_0)\leq \frac{1}{|\psi'(0)|}
$$
for any such $\psi$. Therefore, one has 
\begin{align} \label{eq136-1}
\liminf_{\ell\to \infty} F_{D_{j_\ell}}(z_{j_\ell},X_{j_\ell})\geq F_{M_P}(z_0,X_0)-\delta.
\end{align}

On the other hand, as in \cite{Yu95}, by the tautness of $M_P$, there exists a analytic disc $\varphi \in \mathrm{Hol}(\Delta, M_P)$ such that $\varphi(0)=z_0, \varphi'(0)=\lambda X_0$, where $\lambda=1/F_{M_P}(z_0,X_0)$.

Now for $\delta\in (0,1)$, let us define an analytic disc $\psi_{j_\ell}^\delta:\Delta\to \mathbb C^{n+1}$ by settings:
\begin{align*}
\psi_{j_\ell}^\delta(\zeta):= \varphi((1-\delta)\zeta)+\lambda (1-\delta) (X_{j_\ell}-X_0)+(z_{j_\ell}-z_0)\; \text{for all}\; \zeta \in \Delta.
\end{align*}
Since $\varphi((1-\delta)\overline{\Delta})$ is a compact subset of $M_P$ and $X_{j_\ell}\to X_0$, $z_{j_\ell}\to z_0$ as $\ell\to \infty$, it follows that $\psi_{j_\ell}^\delta(\Delta)\subset D_{j_\ell}$ for all sufficiently large $\ell$, that is, $\psi_{j_\ell}^\delta\in \mathrm{Hol}(\Delta, D_{j_\ell})$. Moreover, by construction, $\psi_{j_\ell}^\delta(0)=z_{j_\ell}$ and $\left(\psi_{j_\ell}^\delta\right)'(0)=(1-\delta)\lambda X_{j_\ell}$. Therefore, again by definition, one has
\begin{align*}
 F_{D_{j_\ell}}(z_{j_\ell},X_{j_\ell})\leq \frac{1}{(1-\delta) \lambda}=\frac{1}{(1-\delta)} F_{M_P}(z_0,X_0)
\end{align*}
for all large $\ell$. Thus, letting $\delta\to 0^+$, one concludes that
\begin{align}\label{eq136-2}
\limsup_{\ell\to \infty} F_{D_{j_\ell}}(z_{j_\ell},X_{j_\ell})\leq  F_{M_P}(z_0,X_0).
\end{align}
By (\ref{eq136-1}), (\ref{eq136-2}), and (\ref{eq136-0}), we seek a contradiction. Hence, the proof is complete.
\end{proof}

\bigskip

\begin{acknowledgement}Part of this work was done while the first author was visiting the Vietnam Institute for Advanced Study in Mathematics (VIASM). He would like to thank the VIASM for financial support and hospitality. The first author was supported by the Vietnam National Foundation for Science and Technology Development (NAFOSTED) under grant number 101.02-2017.311.
\end{acknowledgement}

\bibliographystyle{plain}

\end{document}